\documentclass[12pt,reqno]{article}

\usepackage{amsfonts}
\usepackage{mathrsfs}

\usepackage{amsmath,amsthm,amsfonts,amssymb,latexsym,mathrsfs,color}

\usepackage[colorlinks=true,
linkcolor=webblue, filecolor=webbrown, citecolor=webred ]{hyperref}
\definecolor{webblue}{rgb}{0,.5,0}
\definecolor{webred}{rgb}{0,.5,0}
\definecolor{webbrown}{rgb}{.6,0,0}

\newtheorem{thm}{Theorem}[section]
\newtheorem{lem}[thm]{Lemma}

\theoremstyle{definition}

\newtheorem{rem}[thm]{Remark}

\numberwithin{equation}{section}

\newcommand{\D}{\displaystyle}
\newcommand{\DF}[2]{\frac{\D#1}{\D#2}}

\newcommand{\sep}{\preceq}

\title{On a Stirling-Whitney-Riordan triangle \thanks{Supported partially by the National Natural
Science Foundation of China (Nos. 11971206, 12022105), the Natural
Science Fund for Distinguished Young Scholars of Jiangsu Province
(No. BK20200048) and the Young Talents Invitation Program of
Shandong Province.
\newline\hspace*{5mm}
   {\it Email address:} bxzhu@jsnu.edu.cn (B.-X. Zhu)}}
\author{Bao-Xuan Zhu}
\date{\footnotesize School of Mathematics and Statistics,
         Jiangsu Normal University,
         Xuzhou 221116, PR China}

\begin{document}

\maketitle

\begin{abstract}
Based on the Stirling triangle of the second kind, the Whitney
triangle of the second kind and one triangle of Riordan, we study a
Stirling-Whitney-Riordan triangle $[T_{n,k}]_{n,k}$ satisfying the
recurrence relation:
\begin{eqnarray*}
T_{n,k}&=&(b_1k+b_2)T_{n-1,k-1}+[(2\lambda b_1+a_1)k+a_2+\lambda(
b_1+b_2)] T_{n-1,k}+\\
&&\lambda(a_1+\lambda b_1)(k+1)T_{n-1,k+1},
\end{eqnarray*}
where initial conditions $T_{n,k}=0$ unless $0\le k\le n$ and
$T_{0,0}=1$.

We prove that the Stirling-Whitney-Riordan triangle
$[T_{n,k}]_{n,k}$ is $\textbf{x}$-totally positive with
$\textbf{x}=(a_1,a_2,b_1,b_2,\lambda)$. We show that the
row-generating function $T_n(q)$ has only real zeros and the
Tur\'{a}n-type polynomial $T_{n+1}(q)T_{n-1}(q)-T^2_n(q)$ is stable.
We also present explicit formulae for $T_{n,k}$ and the exponential
generating function of $T_n(q)$ and give a Jacobi continued fraction
expansion for the ordinary generating function of $T_n(q)$.
Furthermore, we get the $\textbf{x}$-Stieltjes moment property and
$3$-$\textbf{x}$-log-convexity of $T_n(q)$ and  show that the
triangular convolution $z_n=\sum_{i=0}^nT_{n,i}x_iy_{n-i}$ preserves
Stieltjes moment property of sequences. Finally, for the first
column $(T_{n,0})_{n\geq0}$, we derive some properties similar to
those of $(T_n(q))_{n\geq0}.$
\end{abstract}

\section{Introduction}
\subsection{Stirling numbers of the second kind}

Let $\left\{
  \begin{array}{ccccc}
    n \\
   k\\
  \end{array}
\right\}$ denote the Stirling number of the second kind. It
enumerates the number of partitions of a set with $n$ elements
consisting of $k$ disjoint nonempty sets. It is well-known that the
Stirling number of the second kind satisfies the recurrence relation
\begin{equation*}
\left\{
  \begin{array}{ccccc}
    n \\
   k\\
  \end{array}
\right\}=k\left\{
  \begin{array}{ccccc}
    n-1 \\
   k\\
  \end{array}
\right\}+\left\{
  \begin{array}{ccccc}
    n-1 \\
   k-1\\
  \end{array}
\right\},
\end{equation*}
where initial conditions $\left\{
  \begin{array}{ccccc}
    0 \\
   0\\
  \end{array}
\right\}=1$ and $\left\{
  \begin{array}{ccccc}
    0 \\
   k\\
  \end{array}
\right\}=0$ for $k\geq1$ or $k<0$. The triangular array
$\left[\left\{
  \begin{array}{ccccc}
    n \\
   k\\
  \end{array}
\right\}\right]_{n,k\geq0}$ is called {\it the Stirling triangle of
the second kind}. Its row-generating function, {\it i.e., the Bell
polynomial}, is defined to be
$
B_n(x)=\sum_{k=0}^{n}\left\{
  \begin{array}{ccccc}
    n \\
   k\\
  \end{array}
\right\} x^{k}. $ There are many nice properties for the Stirling
number and the Bell polynomial. For example:

\begin{itemize}
\item [\rm (i)]
The following identity is well-known
\begin{eqnarray}\label{formu+stirling} \left\{
  \begin{array}{ccccc}
    n \\
   k\\
  \end{array}
\right\}k!=\sum_{j=1}^k(-1)^{k-j}\binom{k}{j}j^n
\end{eqnarray}
for $n,k\geq1$ (see \cite{Com74} for instance).
\item[\rm (ii)]
Let $G_{n,k}=k!\left\{
\begin{array}{ccccc}
    n \\
   k\\
  \end{array}
\right\}$. It counts the number of distinct ordered partitions of a
set with $n$ elements and satisfies the recurrence relation
\begin{eqnarray*}
G_{n,k}=kG_{n-1,k}+kG_{n-1,k-1}.
\end{eqnarray*} Its row-generating function $G_n(x)=\sum_{k=1}^nG_{n,k}x^k$
is called {\it the geometric polynomial} (see Tanny \cite{Tan75}).
\item [\rm (iii)]
The Stirling triangle of the second kind $\left[\left\{
  \begin{array}{ccccc}
    n \\
   k\\
  \end{array}
\right\}\right]_{n,k\geq0}$ is totally positive \cite{Bre95}.
\item [\rm (iv)]
The Bell polynomial $B_n(x)$ has only real zeros and therefore is
log-concave \cite{WYjcta05}.
\item [\rm (v)] The polynomial
$B_{n+1}(x)B_{n-1}(x)-B_n^2(x)$ has no zeros in the right half plane
\cite{Fisk08}.
\item [\rm (vi)]
The exponential generating function for $B_n(x)$ has a concise
formula $ \sum_{n\geq0}B_n(x)\frac{t^n}{n!}=e^{x(e^t-1)} $
\cite{Com74}.
\item[\rm (vii)]
A Jacobi continued fraction expansion related to $B_n(x)$ is given
as
\begin{align*}\sum_{n\geq0}B_n(x)\,t^n=\frac{1}{1-s_0t-\cfrac{r_1t^2}{1-s_1t-\cfrac{r_2t^2}{1-s_2t-\cdots}}},\end{align*}
where $s_n=n+x$ and $r_{n+1}=(n+1)x$ for $n\geq0$ \cite{Fla80}.
\item[\rm (viii)]
The sequence $(B_n(q))_{n\geq0}$ is $q$-log-convex, strongly
$q$-log-convex, $3$-$q$-log-convex and $q$-Stieltjes moment (see
\cite{CWY11,LW07,WZ16,Zhu13,Zhu182} for instance).
\end{itemize}
 We refer the reader to \cite{Com74} for more information of Stirling numbers and Bell
 polynomials.

\subsection{Whitney numbers of the second kind}

 As a generalization of the partition lattice, the Dowling lattice $Q_n(G)$ is a class of geometric lattices based on finite
groups introduced by Dowling \cite{Do73}. The Whitney number of the
second kind, denoted by $W_m(n, k)$, is the number of elements of
corank $k$ of $Q_n(G)$. It satisfies the recurrence relation
\begin{equation*}
W_m(n, k)=(mk+1)W_m(n-1, k)+W_m(n-1, k-1)
\end{equation*}
with $W_m(0,0)=1$. Its row-generating function
$D_n(m,x)=\sum_{k=0}^nW_m(n, k)x^k$ is called {\it the Dowling
polynomial} by Benoumhani \cite{Be97}.

The Whitney number $W_m(n,k)$ (resp. the Dowling polynomial) has
many properties similar to those of the Stirling number $\left\{
  \begin{array}{ccccc}
    n \\
   k\\
  \end{array}
\right\}$ (resp. the Bell polynomial). Let
$W^\diamond_{m}(n,k)=W_{m}(n, k)k!$. It satisfies the recurrence
relation
\begin{equation*}
W^\diamond_{m}(n, k)=(mk+1)W^\diamond_{m}(n-1,
k)+kW^\diamond_{m}(n-1, k-1).
\end{equation*}
Its row-generating function
$F_m(n,x)=\sum_{k=0}^nW^\diamond_{m}(n,k)x^k$ is called {\it the
Tanny-geometric polynomial} in \cite{Be97}. See
\cite{Be97,Be99,CWY11,LW07,WZ16,Zhu13,Zhu182} for some properties of
Whitney numbers and Dowling polynomials such as explicit formulae,
recurrence relations, log-concavity, real rootedness, generating
functions, $q$-log-convexity, strong $q$-log-convexity,
$3$-$q$-log-convexity and $q$-Stieltjes moment property. We also
refer the reader to \cite{CJ12} for more interesting properties
concerning $W_{m}(n, k)$.

\subsection{A triangle of Riordan}
Let $a_{n,k}$ denote the number of set partitions of $[n]$ in which
exactly $k$ of the blocks have been distinguished. It satisfies the
recurrence relation
$$a_{n,k}=a_{n-1,k-1}+(k+1)a_{n-1,k}+(k+1)a_{n-1,k+1},$$
where initial conditions $a_{0,0}=1$ and $a_{n,k}=0$ unless $0\leq
k\leq n$ (see \cite[A049020]{Slo}). Its explicit formula can be
written as
$$a_{n,k}=\sum_{i=0}^n\left\{
  \begin{array}{ccccc}
    n \\
   i\\
  \end{array}
\right\}\binom{i}{k}$$ and $a_{n,0}$ is exactly the famous Bell
number. The triangle $[a_{n,k}]_{n,k}$ first arose in Riordan's
letter \cite{R77} and was also used to characterize Bell number by
Aigner \cite{Aig99}. Its row-generating function can be written in
terms of the next Jacobi continued fraction expansion
\begin{align*}\sum_{n\geq0}\sum_{k=0}^na_{n,k}x^kt^n=\frac{1}{1-s_0t-\cfrac{r_1t^2}{1-s_1t-\cfrac{r_2t^2}{1-s_2t-\cdots}}},\end{align*}
where $s_n=n+1+x$ and $r_{n+1}=(n+1)(x+1)$ for $n\geq0$.  We refer
the reader to \cite[A049020]{Slo} for more information of $a_{n,k}$.

\subsection{Structure of this paper}
The Stirling numbers of the second kind are well-known for their
many nice properties (see \cite{EGH15,HS98,GS78,Zhu14} \cite{}for
instance). Hence the Stirling triangle formed by the Stirling
numbers of the second kind has been widely studied. Note that both
the Whitney triangle and the Riordan triangle mentioned previously
have inner link with the Stirling triangle. The properties similar
to those of the Stirling triangle have also been considered
therefore. Motivated by the attractiveness of studying their
properties, based on the fact of their inner relation, we consider a
new class of triangle generalized form those three. The purpose of
this paper is to provide a unified platform to study the generalized
properties of the new class.

Let $\mathbb{R}$~(resp. $\mathbb{R}^{\geq0}$) denote the set of all
(resp. nonnegative) real numbers. For
$\{\lambda,a_1,a_2,b_1,b_2\}\subseteq \mathbb{R}$, define an array
$T=[T_{n,k}]_{n,k}$, which satisfies the recurrence relation:
\begin{eqnarray*}
T_{n,k}=(b_1k+b_2)T_{n-1,k-1}+[(2\lambda b_1+a_1)k+a_2+\lambda(
b_1+b_2)] T_{n-1,k}+\lambda(a_1+\lambda b_1)(k+1)T_{n-1,k+1},
\end{eqnarray*}
where $T_{0,0}=1$ and $T_{n,k}=0$ unless $0\le k\le n$. Let its
row-generating function $T_n(q)=\sum_{k\geq0}T_{n,k}q^k$ for
$n\geq0$. Obviously, we have
\begin{itemize}
\item
$T_{n,k}=\left\{
  \begin{array}{ccccc}
    n \\
   k\\
  \end{array}
\right\}$ if $a_1=b_2=1$ and $a_2=b_1=\lambda=0$;
\item
   $T_{n,k}=G_{n,k}$ if $a_1=b_1=1,a_2=b_2=\lambda=0$;
   \item
   $T_{n,k}=W_m(n, k)$ if $a_1=m,a_2=b_2=1$ and $b_1=\lambda=0$;
   \item
   $T_{n,k}=W^\diamond_{m}(n, k)$ if $a_1=m,a_2=b_1=1$ and $b_2=\lambda=0$;
   \item
   $T_{n,k}=a_{n,k}$ if $a_1=b_2=\lambda=1$ and $a_2=b_1=0$;
   \item
   $T_{n,k}=\binom{n}{k}k!$ if $b_1=a_2=1$ and $a_1=b_2=\lambda=0$
   (\cite[A008279]{Slo});
   \item
   $T_{n,k}$ is A154602 in \cite{Slo} if $a_1=2$, $b_2=\lambda=1$ and $a_2=b_1=0$.
\end{itemize}
We call this array  $[T_{n,k}]_{n,k}$ {\it a
Stirling-Whitney-Riordan triangle}. The number $T_{n,k}$ can be
interpreted in terms of weighted Motzkin paths due to Flajolet
\cite{Fla80}. Let $u_k=b_1k+b_2+b_1,v_k=[(2\lambda
b_1+a_1)k+a_2+\lambda( b_1+b_2)]$ and $w_k=\lambda(a_1+\lambda
b_1)n$ for $k\geq0$. Then $T_{n,k}$ counts the number of weighted
paths starting from the origin $(0, 0)$ never falling below the
$x$-axis and ending at $(n,k)$ with up diagonal steps $(1, 1)$
weighted $u_{i-1}$, down diagonal steps $(1, -1)$ weighted $w_{i+1}$
and horizontal steps $(1, 0)$ weighted $v_i$ on the line $y=i$.

In Section $2$, we prove that the
 Stirling-Whitney-Riordan triangle $T$ is
$\textbf{x}$-totally positive with
$\textbf{x}=(a_1,a_2,b_1,b_2,\lambda)$. In Section $3$, using the
method of zeros interlacing, we show that $T_n(q)$ has only real
roots and the Tur\'{a}n-type polynomial
$T_{n+1}(q)T_{n-1}(q)-T^2_n(q)$ is stable. In Section $4$, we
present explicit formulae for $T_{n,k}$ and the exponential
generating function of $T_n(q)$. In Section $5$, using addition
formulae of the Stieltjes-Rogers type, we get a Jacobi continued
fraction expansion of the ordinary generating function of $T_n(q)$.
Furthermore, we derive the $\textbf{x}$-Stieltjes moment property
and $3$-$\textbf{x}$-log-convexity of $T_n(q)$ and hence show that
the triangular convolution $z_n=\sum_{k=0}^nT_{n,k}x_ky_{n-k}$
preserves Stieltjes moment property of sequences. Finally, in
Section $6$, for the first column $(T_{n,0})_{n\geq0}$, we derive
some properties similar to those of $(T_n(q))_{n\geq0}.$

\section{Total positivity of the
Stirling-Whitney-Riordan triangle}

Let $A=[a_{n,k}]_{n,k\ge 0}$ be a matrix of real numbers. It is
called {\it totally positive} ({\it TP} for short) if all its minors
are nonnegative. It is called {\it TP$_r$} if all minors of order
$k\le r$ are nonnegative. Let $\textbf{x}=(x_i)_{i\in{I}}$ is a set
of indeterminates. A matrix $A$ with entries being polynomials in
$\mathbb{R}[\textbf{x}]$ is {\it\textbf{x-totally positive}}
(\textbf{x}-TP for short) if all its minors are polynomials with
nonnegative coefficients in the indeterminates $\textbf{x}$ and is
{\it\textbf{x-totally positive of order $r$}} ({\it
\textbf{x}-TP$_r$} for short) if all its minors of order $k\le r$
are polynomials with nonnegative coefficients in the indeterminates
$\textbf{x}$.  Total positivity of matrices is an important and
powerful concept that arises often in various branches of
mathematics (see the monographs \cite{Kar68,Pin10} for general
details about total positivity). We also refer the reader to
\cite{Bre89,Bre95,CLW15,CLW152,GV85,Monge12,S90,Zhu14,Zhu2018,Zhu203}
for total positivity in combinatorics. The following presents the
total positivity of the Stirling-Whitney-Riordan triangle.

\begin{thm}
The Stirling-Whitney-Riordan triangle $T$ is $\textbf{x}$-TP with
$\textbf{x}=(a_1,a_2,b_1,b_2,\lambda)$.
\end{thm}

\begin{proof} Let $\overline{T}$ denote the triangle
obtained from $T$ by deleting its first row. Assume that
\begin{equation*}
J=\left[
\begin{array}{ccccc}
s_0 & r_0 &  &  &\\
t_1 & s_1 & r_1 &\\
 & t_2 & s_2 &r_2&\\
& & \ddots&\ddots & \ddots \\
\end{array}\right],
\end{equation*}
where $r_n=b_1n+b_1+b_2$, $s_n=(2\lambda b_1+a_1)n+a_2+\lambda(
b_1+b_2)$ and $t_n=\lambda(a_1+\lambda b_1)n$.
 The recurrence relation:
\begin{eqnarray*}
T_{n,k}=(b_1k+b_2)T_{n-1,k-1}+[(2\lambda b_1+a_1)k+a_2+\lambda(
b_1+b_2)] T_{n-1,k}+\lambda(a_1+\lambda b_1)(k+1)T_{n-1,k+1}
\end{eqnarray*} implies that
$$\overline{T}=TJ.$$
It follows from \cite[Theorem 2.1]{Zhu202} that $\textbf{x}$-total
positivity of $J$ implies that of $T$. In addition, $J$ is
$\textbf{x}$-TP if and only if
\begin{equation*} J^{*}=\left[
\begin{array}{ccccc}
s_0 & r^{*}_0 &  &  &\\
t^{*}_1 & s_1 & r^{*}_1 &\\
 & t^{*}_2 & s_2 &r^{*}_2&\\
& & \ddots&\ddots & \ddots \\
\end{array}\right]
\end{equation*}
is $\textbf{x}$-TP, where $r^{*}_n=\lambda(b_1n+b_1+b_2)$,
$s_n=(2\lambda b_1+a_1)n+a_2+\lambda( b_1+b_2)$ and
$t^{*}_n=(a_1+\lambda b_1)n$. By \cite[Proposition 3.3 (i)]{Zhu202},
we directly get that $J^{*}$ is $\textbf{x}$-TP with
$\textbf{x}=(a_1,a_2,b_1,b_2,\lambda)$. In consequence, we show that
the triangular matrix $T$ is $\textbf{x}$-TP with
$\textbf{x}=(a_1,a_2,b_1,b_2,\lambda)$. The proof is complete.
\end{proof}

\section{Real rootedness and log-concavity of row-generating functions}
Let $(a_k)_{k\ge 0}$ be a sequence of nonnegative numbers. The
sequence $(a_k)_{k\ge 0}$ is {\it log-concave} if $a_{i-1}a_{i+1}\le
a_i^2$ for $i\geq1$. A basic approach to prove log-concavity is to
use Newton's inequalities: Suppose that the polynomial
$\sum_{k=0}^{n}a_kx^k$ has only real zeros. Then
$$a_k^2\ge a_{k-1}a_{k+1}\left(1+\frac{1}{k}\right)\left(1+\frac{1}{n-k}\right),\quad k=1,2,\ldots,n-1,$$
and $(a_k)_{k\ge 0}$ is therefore log-concave (see Hardy, Littlewood
and P\'olya~\cite[p. 104]{HLP52}). Log-concave sequences and
real-rooted polynomials often occur in combinatorics and have been
extensively investigated. We refer the reader to Brenti
\cite{Bre94}, Stanely \cite{Sta89} and Wang and Yeh \cite{WY07} for
the log-concavity, Br\"and\'en \cite{Bra06,Bra15}, Brenti
\cite{Bre89}, Liu and Wang \cite{LW-RZP}, Wang and Yeh
\cite{WYjcta05} for the real rootedness.

Following Wagner~\cite{Wag92}, a real polynomial is said to be {\it
standard} if either it is identically zero or its leading
coefficient is positive. Assume that both polynomials $f$ and $g$
only have real zeros. Let $\{r_i\}$ and $\{s_j\}$ be all zeros of
$f$ and $g$ in nondecreasing order respectively. We say that $g$
{\it interlaces} $f$ denoted by $g\sep f$ if $\deg f=\deg g+1=n$ and
\begin{equation}\label{int-def}
r_n\le s_{n-1}\le\cdots\le s_2\le r_2\le s_1\le r_1.
\end{equation}

For two interlacing polynomials, Fisk showed the following result.
\begin{lem}\emph{\cite[Lemma 1.20]{Fisk08}}\label{lem+Fisk}
Let both $f(x)$ and $g(x)$ be standard real polynomials with only
real zeros. Assume that $\deg(f(x))=n$ and all real zeros of $f(x)$
are $s_1, \ldots, s_n$. If $\deg(g)=n-1$ and we write
$$g(x)=\sum_{i=1}^n\frac{c_if(x)}{x-s_i},$$
then $g\sep f$ if and only if all $c_i$ are positive.
\end{lem}

A real polynomial is {\it weakly (Hurwitz) stable} if all of its
zeros lie in the closed left half of the complex plane. See
\cite[Chapter 9]{M66} for deep surveys on the stability theory of
polynomials. Let $L_n(x)$ be the $n$th Legendre polynomial.
Tur\'{a}n-type inequalities \cite{T50} state that
$$L^2_n(x)-L_{n+1}(x)L_{n-1}(x)>0\,\, \text{for}\, -1<x<1.
$$
In 1948, Szeg\"{o} gave four different proofs of the famous
 Tur\'{a}n-type inequality on Legendre polynomials \cite{S48}. It has been proved that many important (orthogonal) polynomials and special functions
 satisfy some Tur\'{a}n-type inequalities (see
\cite{BK16} for instance).

The following gives the zeros properties related to $T_{n}(x)$.

\begin{thm}\label{thm+RZ}
Let $T_n(x)$ be the row-generating function of the
Stirling-Whitney-Riordan triangle $T$. If
$\{\lambda,a_1,a_2,b_1,b_2\}\subseteq \mathbb{R}^{\geq0}$ and
$a_1(b_1+b_2)> b_1a_2$\footnote{\quad If $a_1(b_1+b_2)\geq b_1a_2$,
then zeros of $T_n(q)$ may be not simple.}, then
\begin{itemize}
  \item [\rm (i)]
$T_n(x)$ has only simple real zeros  in
$(-\lambda-\frac{a_1}{b_1},-\lambda)$\quad\footnote{\quad If
$b_1=0$, then $-\lambda-\frac{a_1}{b_1}$ means $-\infty$.}
  and $T_{n-1}(x)\sep T_n(x)$ for $n\geq1$. Therefore $T_n(x)$ is log-concave for $n\geq1$.
  \item [\rm (ii)]
The Tur\'{a}n-type polynomial $T_{n+1}(x)T_{n-1}(x)-(T_n(x))^2$ is a
weakly stable polynomial for $n\geq1$.
\end{itemize}
\end{thm}
\begin{proof}
(i)  It follows from the recurrence relation:
\begin{eqnarray*}
T_{n,k}=(b_1k+b_2)T_{n-1,k-1}+[(2\lambda b_1+a_1)k+a_2+\lambda(
b_1+b_2)] T_{n-1,k}+\lambda(a_1+\lambda b_1)(k+1)T_{n-1,k+1}
\end{eqnarray*}
that
\begin{eqnarray}\label{rec+Tx}
T_{n}(x)=[a_2+(b_1+b_2)(x+\lambda)]T_{n-1}(x)+(x+\lambda)[a_1+b_1(x+\lambda)]T'_{n-1}(x)
\end{eqnarray}
for $n\geq1$. We will prove by induction on $n$ that $T_n(x)$ has
only simple real zeros in $(-\lambda-\frac{a_1}{b_1},-\lambda)$ and
$T_{n-1}(x)\sep T_n(x)$ for $n\geq1$. Clearly, $T_0(x)=1$.

Case 1: Assume $b_1\neq0$. For $n=1$, we have
$$T_1(x)=a_2+(b_1+b_2)(x+\lambda).$$ Obviously, it follows from $a_1(b_1+b_2)> b_1a_2$ that $T_1(x)$ has only real zero
in $(-\lambda-\frac{a_1}{b_1},-\lambda)$. Suppose for $n\geq2$ that
$T_{n-1}(x)$ has $n-1$ real zeros denoted by $$ -\lambda> s_1> s_2
>\ldots > s_{n-1}>-\lambda-\frac{a_1}{b_1}.$$ Then by the recurrence relation \eqref{rec+Tx}, we have
$$sign[ T_n(s_k)]=(-1)^k.$$ In consequence, $T_n(x)$ has $n$ simple real zeros
denoted by $r_1> r_2 >\ldots > r_n$ such that
\begin{eqnarray}\label{order+root}
r_1> s_1> r_2 > s_2>\ldots > s_{n-1}> r_n.
\end{eqnarray}

On the other hand, by the recurrence relation (\ref{rec+Tx}), we
have
$$T_n(-\lambda)=a_2T_{n-1}(-\lambda)>0,\quad T_n(-\lambda-\frac{a_1}{b_1})=\frac{a_2b_1-a_1(b_1+b_2)}{b_1}T_{n-1}(-\lambda-\frac{a_1}{b_1}),$$
which imply $r_1<-\lambda$ and $r_n>-\lambda-\frac{a_1}{b_1}$.

Case 2: Assume $b_1=0$. For $n=1$, we have
$$T_1(x)=a_2+b_2(x+\lambda).$$
Obviously, it follows from $a_1b_2> b_1a_2$ that $T_1(x)$ has only
real zero in $(-\infty,-\lambda)$. Suppose for $n\geq2$ that
$T_{n-1}(x)$ has $n-1$ real zeros denoted by $$ -\lambda> s_1> s_2
>\ldots > s_{n-1}.$$ Then by the recurrence relation \eqref{rec+Tx}, we have
$$sign[ T_n(s_k)]=(-1)^k.$$ This implies that $T_n(x)$ has $n$ simple real zeros
denoted by $r_1> r_2 >\ldots > r_n$ such that
\begin{eqnarray}
 r_1> s_1> r_2 > s_2>\ldots > s_{n-1}> r_n.
\end{eqnarray}

In addition, by the recurrence relation (\ref{rec+Tx}), we have
$$T_n(-\lambda)=a_2T_{n-1}(-\lambda)>0.$$
Thus $r_1<-\lambda$. This completes the proof of (i).

(ii) By \eqref{rec+Tx}, we deduce that
\begin{eqnarray}\label{stable}
&&T_{n+1}(x)T_{n-1}(x)-(T_n(x))^2\nonumber\\
&=&[a_2+(b_1+b_2)(x+\lambda)]T_n(x)T_{n-1}(x)+(x+\lambda)[a_1+b_1(x+\lambda)]T'_{n}(x)T_{n-1}(x)-\nonumber\\
&&[a_2+(b_1+b_2)(x+\lambda)]T_n(x)T_{n-1}(x)-(x+\lambda)[a_1+b_1(x+\lambda)]T'_{n-1}(x)T_{n}(x) \nonumber\\
&=&(x+\lambda)[a_1+b_1(x+\lambda)]\left[T'_{n}(x)T_{n-1}(x)-T'_{n-1}(x)T_{n}(x)\right]\nonumber\\
&=&-(x+\lambda)[a_1+b_1(x+\lambda)][T_n(x)]^2\left(\frac{T_{n-1}(x)}{T_{n}(x)}\right)'\label{poly}.
\end{eqnarray}

By (i), assume that $T_n(x)$ has $n$ negative zeros as $ r_1> r_2
>\ldots > r_n$. It follows from $T_{n-1}(x)\sep T_n(x)$ in (i) and
Lemma~\ref{lem+Fisk} that we get
\begin{eqnarray}\label{rat}
\frac{T_{n-1}(x)}{T_n(x)}=\sum_{k=1}^n\frac{t_k}{x-r_k},
\end{eqnarray}
where $t_k> 0$ for $1\leq k\leq n$. Combining \eqref{stable} and
(\ref{rat}) gives the iterated Tur\'{a}n-type polynomial
\begin{eqnarray*}
T_{n+1}(x)T_{n-1}(x)-(T_n(x))^2
&=&(x+\lambda)[a_1+b_1(x+\lambda)][T_n(x)]^2\sum_{k=1}^n\frac{t_k}{(x-r_k)^2}.
\end{eqnarray*}

Thus, in order get that $T_{n+1}(x)T_{n-1}(x)-(T_n(x))^2$ is sable,
it suffices to prove that
$$\sum_{k=1}^n\frac{t_k}{(x-r_k)^2}\neq0$$ for $x$ in the right half plane. Obviously,
$\sum_{k=1}^n\frac{t_k}{(x-r_k)^2}>0$ for $x\geq0$. In addition, for
$x=u+vi$ with $u>0$ and $v\neq0$,
\begin{align*}
&Im\left(\sum_{k=1}^n\frac{t_k}{(x-r_k)^2}\right)=-2v\sum_{k=1}^n\frac{t_k(u-r_k)}{(u-r_k)^2+v^2}\neq0
\end{align*}
since $\frac{t_k(u-r_k)}{(u-r_k)^2+v^2}>0$ for each $k\in[1,n]$.
 In consequence, we get
that
$$\sum_{k=1}^n\frac{t_k}{(x-r_k)^2}$$ has no zeros in the right half plane. So does the
polynomial $T_{n+1}(x)T_{n-1}(x)-(T_n(x))^2$. We complete the proof
of (ii).
\end{proof}

\section{Exponential generating function and explicit formula }
In this section, we will present the exponential generating function
of $T_n(q)$ and use it to derive an explicit formula for the
Stirling-Whitney-Riordan triangle $T$ as follows.
\begin{thm}\label{thm+EXP}
Let $T_n(q)$ be the row-generating function of $T$ with
$a_1^2+b_1^2\neq0$.
\begin{itemize}
  \item [\rm (i)]
The exponential generating function of $T_n(q)$ is given as
\footnote{\quad For $a_1=0$ or $b_1=0$, using continuity of
functions, the corresponding formula for exponential generating
function means its limits as follows:
\begin{eqnarray*}
\sum_{n\geq0}T_n(q)\frac{t^n}{n!}&=&\begin{cases}
e^{a_2t}\left[1-b_1(q+\lambda)t\right]^{-(1+\frac{b_2}{b_1})},&\text{for}\quad
a_1=0,b_1\neq0\\
&\\
e^{a_2t+\left[\frac{b_2(q+\lambda)(e^{a_1t}-1)}{a_1}\right]},&\text{for}\quad
a_1\neq0,b_1=0.
\end{cases}
\end{eqnarray*}}
 \begin{eqnarray*}\label{EGF+F}
\sum_{n\geq0}T_n(q)\frac{t^n}{n!}&=&
e^{a_2t}\left[1+\frac{b_1(q+\lambda)(1-e^{a_1t})}{a_1}\right]^{-(1+\frac{b_2}{b_1})}.
\end{eqnarray*}
  \item [\rm (ii)]
An explicit formula for $T_{n,k}$ can be written as
  \begin{eqnarray*}
T_{n,k}&=&\begin{cases} \sum_{i\geq
k}\frac{\prod_{j=1}^i(b_2+b_1j)}{a^i_1}\times\binom{i}{k}{\lambda}^{i-k}\times\frac{1}{i!}\sum_{j=0}^i\binom{i}{j}(-1)^{i-j}(a_2+a_1j)^n,
&\text{for}\quad a_1\neq0\\
\sum_{i\geq
k}\prod_{j=1}^i(b_2+b_1j)\times\binom{n}{i}\binom{i}{k}{\lambda}^{i-k}a^{n-i}_2,&\text{for}\quad
a_1=0.
\end{cases}
\end{eqnarray*}
\end{itemize}
\end{thm}

\begin{proof}
Let the exponential generating function
$$\mathcal {T}(q,t)=\sum_{n,k\geq0}T_{n,k}q^k\frac{t^n}{n!}=\sum_{n\geq0}T_n(q)\frac{t^n}{n!}.$$
Then by the recurrence relation:
\begin{eqnarray*}
T_{n,k}=(b_1k+b_2)T_{n-1,k-1}+[(2\lambda b_1+a_1)k+a_2+\lambda(
b_1+b_2)] T_{n-1,k}+\lambda(a_1+\lambda b_1)(k+1)T_{n-1,k+1},
\end{eqnarray*} we have the next
partial differential equation
$$\mathcal {T}_t(q,t)-[b_1(q+\lambda)+a_1](q+\lambda)\mathcal {T}_q(q,t)=[a_2+(b_1+b_2)(q+\lambda)]\mathcal {T}(q,t)$$
with the initial condition $\mathcal {T}(q,0)=1$. It is routine to
check that
\begin{eqnarray}\label{EGF+F}
\mathcal
{T}(q,t)=e^{a_2t}\left[1+\frac{b_1(q+\lambda)(1-e^{a_1t})}{a_1}\right]^{-(1+\frac{b_2}{b_1})}
\end{eqnarray} is a solution of the above
partial differential with the initial condition.

(ii) Its proof will be divided into the following three cases.

Case $1$: $a_1b_1\neq0$. We have
\begin{eqnarray*}
 \sum_{n,k\geq0}
 T_{n,k}q^k\frac{t^n}{n!}&=&e^{a_2t}\left[1+\frac{b_1(q+\lambda)(1-e^{a_1t})}{a_1}\right]^{-(1+\frac{b_2}{b_1})}\\
 &=&e^{a_2t}\left[1-\frac{b_1(q+\lambda)(e^{a_1t}-1)}{a_1}\right]^{-(1+\frac{b_2}{b_1})}\\
 &=&e^{a_2t}\sum_{i\geq0}\binom{\frac{b_2}{b_1}+i}{i}\left[\frac{b_1(q+\lambda)(e^{a_1t}-1)}{a_1}\right]^i\\
 &=&e^{a_2t}\sum_{i\geq0}\binom{\frac{b_2}{b_1}+i}{i}\left(\frac{b_1(q+\lambda)}{a_1}\right)^i\sum_{j=0}^i\binom{i}{j}(-1)^{i-j}e^{a_1jt}\\
 &=&\sum_{i\geq0}\binom{\frac{b_2}{b_1}+i}{i}\left(\frac{b_1(q+\lambda)}{a_1}\right)^i\sum_{j=0}^i\binom{i}{j}(-1)^{i-j}e^{(a_2+a_1j)t}\\
 &=&\sum_{i\geq0}\frac{\prod_{j=1}^i(b_2+b_1j)}{i!}\times\frac{(q+\lambda)^i}{a^i_1}\sum_{j=0}^i\binom{i}{j}(-1)^{i-j}\sum_{n\geq0}\frac{(a_2+a_1j)^nt^n}{n!},
\end{eqnarray*}
which implies
\begin{eqnarray*}
 T_{n,k}=\sum_{i\geq k}\frac{\prod_{j=1}^i(b_2+b_1j)}{a^i_1}\times\binom{i}{k}{\lambda}^{i-k}\times\frac{1}{i!}\sum_{j=0}^i\binom{i}{j}(-1)^{i-j}(a_2+a_1j)^n.
\end{eqnarray*}
Case $2$: $b_1=0$ and $a_1\neq0$. We have
\begin{eqnarray*}
 \sum_{n,k\geq0}
 T_{n,k}q^k\frac{t^n}{n!}&=&e^{a_2t+\left[\frac{b_2(q+\lambda)(e^{a_1t}-1)}{a_1}\right]}\\
 &=&e^{a_2t}\sum_{i\geq0}\left[\frac{b_2(q+\lambda)(e^{a_1t}-1)}{a_1}\right]^i\times\frac{1}{i!}\\
 &=&e^{a_2t}\sum_{i\geq0}\left(\frac{b_2(q+\lambda)}{a_1}\right)^i\times\frac{1}{i!}\sum_{j=0}^i\binom{i}{j}(-1)^{i-j}e^{a_1jt}\\
 &=&\sum_{i\geq0}\left(\frac{b_2}{a_1}\right)^i(q+\lambda)^i\times\frac{1}{i!}\sum_{j=0}^i\binom{i}{j}(-1)^{i-j}e^{(a_2+a_1j)t}\\
 &=&\sum_{i\geq0}\left(\frac{b_2}{a_1}\right)^i(q+\lambda)^i\times\frac{1}{i!}\sum_{j=0}^i\binom{i}{j}(-1)^{i-j}\sum_{n\geq0}\frac{(a_2+a_1j)^nt^n}{n!}.
\end{eqnarray*}
Obviously,
\begin{eqnarray*}
 T_{n,k}=\sum_{i\geq k}\frac{b^i_2}{a^i_1}\binom{i}{k}{\lambda}^{i-k}\times\frac{1}{i!}\sum_{j=0}^i\binom{i}{j}(-1)^{i-j}(a_2+a_1j)^n.
\end{eqnarray*}
Case $3$: $a_1=0$ and $b_1\neq0$. We have
\begin{eqnarray*}
 \sum_{n,k\geq0}
 T_{n,k}q^k\frac{t^n}{n!}&=&e^{a_2t}\left[1-b_1(q+\lambda)t\right]^{-(1+\frac{b_2}{b_1})}\\
 &=&e^{a_2t}\sum_{i\geq0}\binom{\frac{b_2}{b_1}+i}{i}\left[b_1(q+\lambda)t\right]^i\\
 &=&e^{a_2t}\sum_{i\geq0}\prod_{j=1}^i(b_2+b_1j)\times(q+\lambda)^i\frac{t^i}{i!}.
\end{eqnarray*}
Thus, we get
\begin{eqnarray*}
 T_{n,k}=\sum_{i\geq k}\prod_{j=1}^i(b_2+b_1j)\times\binom{n}{i}\binom{i}{k}{\lambda}^{i-k}a^{n-i}_2.
\end{eqnarray*}

This completes the proof.
\end{proof}

\section{Stieltjes moment property and continued fractions}

In this section, we will present a continued fraction expansion of
$\sum_{n\geq0}T_n(q)t^n$ and demonstrate a Stieltjes moment property
for $T_n(q)$. Continued fraction is an important tool in
combinatorics, which is closely related to many aspects, {\it e.g.,}
combinatorial lattice paths, combinatorial interpretations,
combinatorial identities, combinatorial positivity, determinants of
sequences, and so on. We refer the reader to Flajolet \cite{Fla80}
for more information concerning continued fraction expansions
related to many important combinatorial objects.

Continued fraction plays an important role in studying Hankel-total
positivity and Stieltjes moment sequences. Given a sequence
$\alpha=(a_k)_{k\ge 0}$, define its {\it Hankel matrix} $H(\alpha)$
by
$$H(\alpha)=[a_{i+j}]_{i,j\ge 0}=
\left[
  \begin{array}{ccccc}
    a_0 & a_1 & a_2 & a_3 & \cdots \\
    a_1 & a_2 & a_3 & a_4 & \cdots \\
    a_2 & a_3 & a_4 & a_5 &\cdots  \\
    a_3 & a_4 & a_5 & a_6 & \cdots\\
    \vdots &\vdots  & \vdots & \vdots & \ddots \\
  \end{array}
\right].$$ We say that $\alpha$ is a {\it Stieltjes moment} ({\it
SM} for short) sequence if it has the form
\begin{equation}\label{i-e}
a_k=\int_0^{+\infty}x^kd\mu(x),
\end{equation}
where $\mu$ is a non-negative measure on $[0,+\infty)$ (see
\cite[Theorem 4.4]{Pin10} for instance). The Stieltjes moment
problem is one of classical moment problems and arises naturally in
many branches of mathematics \cite{ST43,Wid41}. It is well-known
that the following are equivalent:
\begin{itemize}
  \item [\rm (i)]
$\alpha$ is a Stieltjes moment sequence.
  \item [\rm (ii)]
Its Hankel matrix $H(\alpha)$ is TP.
 \item [\rm (iii)]
Its generating function has the Stieltjes continued fraction
expansion
$$\sum_{n\geq0}a_nz^n=\frac{1}{1-\cfrac{\beta_0z}{1-\cfrac{\beta_1z}{1-\cfrac{\beta_2z}{1-\cdots}}}}$$
with $\beta_i\geq0$ for $i\geq0$.
 \item [\rm (iv)]
 Positivity characterization: $\sum_{n=0}^Nc_na_n\ge 0$ for every polynomial
$ \sum_{n=0}^Nc_nq^n\ge 0 $ on $[0,\infty)$.
\end{itemize}
Let $\textbf{x}=(x_i)_{i\in{I}}$ is a set of indeterminates. A
polynomial sequence $(\alpha_n(\textbf{x}))_{n\geq0}$ in
$\mathbb{R}[\textbf{x}]$ is called a {\it $\textbf{x}$-Stieltjes
moment} ($\textbf{x}$-SM for short) sequence if its associated
infinite Hankel matrix is $\textbf{x}$-TP, see Zhu
\cite{Zhu2018,Zhu202} for instance. When
$(\alpha_n(\textbf{x}))_{n\geq0}$ is a sequence of real numbers,
$\textbf{x}$-SM sequence reduces to the classical Stieltjes moment
sequence. For $\textbf{x}$-SM sequences, the following criterion was
proved in \cite{Zhu2018,Zhu202}.

\begin{lem}\emph{\cite{Zhu202}}\label{lem+q-SM}
Let $\{s_n(\textbf{x}),r_{n}(\textbf{x}),H_n(\textbf{x})\}\subseteq
\mathbb{R}^{\geq0}[\textbf{x}]$ for $n\in \mathbb{N}$ and
\begin{align*}\sum_{n\geq0}
H_n(\textbf{x})t^n=\frac{1}{1-s_0(\textbf{x})t-\cfrac{r_1(\textbf{x})t^2}{1-s_1(\textbf{x})t-\cfrac{r_2(\textbf{x})t^2}{1-\cdots}}}.
\end{align*}
If there exists
$\{\lambda_n(\textbf{x}),u_n(\textbf{x}),v_n(\textbf{x})\}\subseteq\mathbb{R}^{\geq0}[\textbf{x}]$
such that $s_n=\lambda_n+u_n+v_n$ and $r_{n+1}=u_{n+1}v_{n}$ for
$n\geq0$, then polynomials $H_n(\textbf{x})$ form a $\textbf{x}$-SM
sequence for $n\geq0$.
\end{lem}
 In order to compute continued
fraction, we need the following addition formulae of the
Stieltjes-Rogers type.
\begin{lem}\cite{R07,S89}\label{lem+S+R}
For a sequence $(\alpha_n)_{n\geq0}$, define the function
  \begin{eqnarray*}
  h(x)&=&\sum_{n\geq0}\alpha_n\frac{x^n}{n!}.
  \end{eqnarray*}
 If there exists two sequences $(e_{n})_{n\geq0}$ and $(w_n)_{n\geq0}$ such that the generating
function
$$h(x+y)=\sum_{n\geq0}w_kf_k(x)f_k(y),$$
where
$$f_k(x)=\frac{x^k}{k!}+e_{k+1}\frac{x^{k+1}}{(k+1)!}+O(x^{k+2}),$$
then we have
\begin{align*}
\sum_{n\geq0}\alpha_n
t^n=\frac{1}{1-s_0t-\cfrac{r_1t^2}{1-s_1t-\cfrac{r_2t^2}{1-s_2t-\cfrac{r_3t^2}{1-s_3t-\cdots}}}},
\end{align*}
where $s_n=e_{n+1}-e_n$ and $r_{n+1}=w_{n+1}/w_n$ for $n\geq0$.
\end{lem}

If a polynomial sequence $(A_n(q))_{n\geq0}$ in an indeterminate $q$
is $q$-SM, then its triangular convolution preserves the SM property
in terms of the next result.

\begin{lem}\label{lem+conv}\emph{\cite{WZ16}}
For $n\in \mathbb{N},$ let $A_n(q)=\sum_{k=0}^{n}A_{n,k}q^k$ be the
$n$th row-generating function of a matrix $[A_{n,k}]_{n,k\geq0}$.
Assume that $(A_n(q))_{n\ge 0}$ is a SM sequence for any fixed $q\ge
0$. If both $(x_n)_{n\ge 0}$ and $(y_n)_{n\ge 0}$ are SM sequences,
then so is $(z_n)_{n\ge 0}$ defined by
\begin{eqnarray}\label{a-c}
z_n=\sum_{k=0}^{n}A_{n,k}x_ky_{n-k}.
\end{eqnarray}
\end{lem}

For the $\textbf{x}$-SM property, one necessary condition is
$\textbf{x}$-log-convexity. For a polynomial sequence
$(f_n(\textbf{x}))_{n\geq 0}$, it is {\it $\textbf{x}$-log-convex}
if
$$ f_{n+1}(\textbf{x})f_{n-1}(\textbf{x})- f_n(\textbf{x})^2$$ is a polynomial with nonnegative coefficients for $n\geq 1$. Define the operator
$\mathcal {L}$ which maps a polynomial sequence
$(f_n(\textbf{x}))_{n\geq 0}$ to another polynomial sequence
$(g_i(\textbf{x}))_{i\geq 1}$ given by
$$g_i(\textbf{x}):=f_{i-1}(\textbf{x})f_{i+1}(\textbf{x})-f_i(\textbf{x})^2.$$
Then the $\textbf{x}$-log-convexity of $(f_n(\textbf{x}))_{n\geq 0}$
is equivalent to the $\textbf{x}$-positivity of $\mathcal
{L}\{f_i(\textbf{x})\}$, i.e., the coefficients of $g_i(\textbf{x})$
are nonnegative for all $i\geq1$. Generally, we say that
$(f_i(\textbf{x}))_{i\geq 0}$ is {\it $k$-$\textbf{x}$-log-convex}
if the coefficients of $\mathcal {L}^m\{f_i(\textbf{x})\}$ are
nonnegative for all $m\leq k$, where $\mathcal {L}^m=\mathcal
{L}(\mathcal {L}^{m-1})$.

\begin{lem}\emph{\cite{Zhu202}}\label{lem+3-q-log-convex}
If the Hankel matrix $[A_{i+j}(\textbf{x})]_{i,j\geq0}$ is
$\textbf{x}$-TP$_4$, then the sequence $(A_n(\textbf{x}))_{n\geq0}$
is $3$-$\textbf{x}$-log-convex.
\end{lem}
Obviously, if $(A_n(\textbf{x}))_{n\geq0}$ is a $\textbf{x}$-SM
sequence, then $[A_{i+j}(\textbf{x})]_{i,j}$ is
$\textbf{x}$-Hankel-TP$_4$. Thus it is $3$-$\textbf{x}$-log-convex
by Lemma \ref{lem+3-q-log-convex}.

If $\textbf{x}$ is an indeterminate $q$, then it has been proved
that many famous polynomials have the $q$-log-convexity, e.g., the
Bell polynomials, the classical Eulerian polynomials, the Narayana
polynomials of type $A$ and $B$, Jacobi-Stirling polynomials, and so
on (see Liu and Wang \cite{LW07}, Chen {\it et al.} \cite{CTWY10},
Zhu \cite{Zhu13,Zhu14,Zhu17,Zhu182,Zhu20} for instance). These
polynomials also have $3$-$q$-log-convexity (see Zhu
\cite{Zhu18,Zhu19}).

The main result of this section is the following.

\begin{thm}\label{thm+SM+CT}
Let $T_n(q)$ be the row-generating function of $T$. Then we have the
next results.
\begin{itemize}
  \item [\rm (i)]
The ordinary generating function of $T_n(q)$ has a Jacobi continued
  fraction expression

  $$\sum_{n\geq0}T_n(q)t^n=\frac{1}{1-s_0t-\cfrac{r_1t^2}{1-s_1t-\cfrac{r_2t^2}{1-s_2t-\cdots}}},$$
where $s_n=a_2+a_1n+[b_1(2n+1)+b_2](q+\lambda)$ and
$r_{n+1}=[b_1(n+1)+b_2](q+\lambda)[b_1(q+\lambda)+a_1](n+1)$ for
$n\geq0$.
   \item [\rm (ii)]
 The sequence $(T_{n}(q))_{n\geq0}$ is $\textbf{x}$-SM
and $3$-$\textbf{x}$-log-convex with
$\textbf{x}=(a_1,a_2,b_1,b_2,\lambda,q)$.
  \item [\rm (iii)]
The convolution $z_n=\sum_{k\geq0}T_{n,k}x_ky_{n-k}$ preserves SM
property if $\{\lambda,a_1,a_2,b_1,b_2\}\subseteq
\mathbb{R}^{\geq0}$.
\item[\rm(iv)] We have Hankel-determinants
\begin{eqnarray*} \det_{0\leq i,j\leq
n-1}(T_{i+j}(q))&=&r_1^{n-1}r_2^{n-2}\ldots r_{n-2}^2r_{n-1}
\end{eqnarray*}  and
\begin{eqnarray*}
\det_{0\leq i,j\leq n-1}(T_{i+j+1}(q))&=&r_1^{n-1}r_2^{n-2}\cdots
r_{n-2}^2r_{n-1}Q_n,
\end{eqnarray*}
where $(Q_n)_{n\geq0}$ is defined by $Q_{n+1}=s_nQ_n-r_nQ_{n-1}$
with $Q_0=1$ and $Q_1=s_0$.
\end{itemize}
\end{thm}
\begin{proof}
Let the exponential generating function
$$G(q,t)=\sum_{n\geq0} T_n(q)\frac{t^n}{n!}.$$  Then
\begin{eqnarray*}
G(q,t)=e^{a_2t}\left[1+\frac{b_1(q+\lambda)(1-e^{a_1t})}{a_1}\right]^{-(1+\frac{b_2}{b_1})}.
\end{eqnarray*}
In the following, we only need to consider the case $a_1b_1\neq0$
(for the case $a_1b_1=0$, it is the corresponding limits in terms of
continuity). Let $\gamma=\frac{b_1}{a_1}$, $\beta=\frac{b_2}{b_1}$
and $p=q+\lambda$. Assume that
$$h(x+y)=e^{(x+y)a_2}\left[1-\gamma p(e^{a_1(x+y)}-1)\right]^{-(1+\beta)}.$$
Then\begin{eqnarray*} h(x+y)&=&e^{(x+y)a_2}\left\{[1-\gamma p(e^{a_1x}-1)][1-\gamma p(e^{a_1y}-1)]-\gamma p(1+\gamma p)(e^{a_1y}-1)(e^{a_1x}-1)\right\}^{-(1+\beta)}\\
&=&\sum_{k\geq0}k!a_1^{2k}\langle 1+\beta \rangle_k(\gamma
p)^k(1+\gamma p)^kf_k(x)f_k(y),
\end{eqnarray*}
where $\langle 1+\beta \rangle_k=(1+\beta)(2+\beta)\cdots(k+\beta)$
and
\begin{eqnarray*}
f_k(x)&=&\frac{1}{k!a_1^k}e^{xa_2}(e^{a_1x}-1)^k\left[1-\gamma p(e^{a_1x}-1)\right]^{-(1+\beta+k)}\\
&=&\frac{1}{k!a_1^k}\left(1+a_2 x+\frac{(a_2 x)^2}{2}+\cdots\right)\left(a_1x+\frac{a^2_1x^2}{2}+\cdots\right)^k\left[1+\gamma p(1+\beta)a_1x+\cdots\right]\\
&=&\frac{x^k}{k!}+\frac{x^{k+1}}{(k+1)!}\left[a_2+\frac{a_1k}{2}+(1+\beta+k)\gamma
pa_1\right](k+1)+O(x^{k+2}).
\end{eqnarray*}
By Lemma \ref{lem+S+R}, we have
$$w_k=k!a_1^{2k}\langle1+\beta\rangle_k(\gamma p)^k(1+\gamma p)^k, \, \quad
e_{k+1}=\left[a_2+\frac{a_1k}{2}+(1+\beta+k)\gamma
pa_1\right](k+1).$$ Thus, we get
$$s_k=e_{k+1}-e_k=a_2+a_1k+(2k+1+\beta)\gamma pa_1, \, \quad
r_{k+1}=\frac{w_{k+1}}{w_k}=a_1^2(k+1+\beta)\gamma p(\gamma
p+1)(k+1)$$ for $k\geq0$. So
$$\sum_{n\geq0}T_n(q)t^n=\frac{1}{1-s_0t-\cfrac{r_1t^2}{1-s_1t-\cfrac{r_2t^2}{1-s_2t-\cdots}}}.$$
where $s_n=a_2+a_1n+[b_1(2n+1)+b_2](q+\lambda)$ and
$r_{n+1}=[b_1(n+1)+b_2](q+\lambda)[b_1(q+\lambda)+a_1](n+1)$ for
$n\geq0$.

Let $v_{n}=(nb_1+b_2+b_1)(q+\lambda)$ and
$u_{n}=n[a_1+b_1(q+\lambda)]$ for $n\geq0$. It is obvious that
$s_n=a_2+u_{n}+v_{n}$ and $t_n=v_{n}u_{n+1}$ for $n\geq0$. It
follows from Lemma \ref{lem+q-SM} that $(T_n(q))_{n\geq0}$ is a
$\textbf{x}$-SM sequence with
$\textbf{x}=(a_1,a_2,b_1,b_2,\lambda,q)$. Then $[T_{i+j}(q)]_{i,j}$
is $\textbf{x}$-TP$_4$. Thus by Lemma \ref{lem+3-q-log-convex}, we
immediately have $(T_n(q))_{n\geq0}$ is $3$-$\textbf{x}$-log-convex.
In addition, it follows from Lemma \ref{lem+conv} that the
Stirling-Whitney-Riordan triangle convolution
$$z_n=\sum_{k=0}^{n}T_{n,k} x_ky_{n-k},\quad n=0,1,2,\ldots$$
preserves the SM property. Finally, for (iv), it follows from the
next general criterion (see \cite{MWY17} for instance): If the
generating function of $(u_i)_{i\geq0}$ can be expressed by
\begin{eqnarray*} \sum\limits_{i=0}^{\infty}u_ix^i=\DF{u_0}{1-
s_0x-\DF{t_1x^2}{1-s_1x-\DF{t_2x^2}{1-s_2x-\ldots}}},
\end{eqnarray*}
then \begin{eqnarray*} \det_{0\leq i,j\leq
n-1}(u_{i+j})&=&u_0^nt_1^{n-1}t_2^{n-2}\ldots t_{n-2}^2t_{n-1}
\end{eqnarray*}  and
\begin{eqnarray*}
\det_{0\leq i,j\leq n-1}(u_{i+j+1})&=&u_0^nt_1^{n-1}t_2^{n-2}\cdots
t_{n-2}^2t_{n-1}q_n,
\end{eqnarray*}
where $(q_n)_{n\geq0}$ is defined by $q_{n+1}=s_nq_n-t_nq_{n-1}$
with $q_0=1$ and $q_1=s_0$.

\end{proof}

\begin{rem}
Note that for a Jacobi continued fraction expansion, it has a
general combinatorial interpretation in terms of weighted Motzkin
paths due to Flajolet \cite{Fla80}.  For the Jacobi continued
fraction expansion
\begin{align*}
\sum_{n\geq0}T_n(q)\,t^n=\frac{1}{1-s_0t-\cfrac{r_1t^2}{1-s_1t-\cfrac{r_2t^2}{1-s_2t-\cfrac{r_3t^2}{1-s_3t-\cdots}}}}
\end{align*}
with $s_n=a_2+a_1n+[b_1(2n+1)+b_2](q+\lambda)$ and
$r_{n+1}=(n+1)[b_1(n+1)+b_2](q+\lambda)[b_1(q+\lambda)+a_1]$ for
$n\geq0$, we can interpret $T_n(q)$ as follows: weighted Motzkin
paths start from the origin $(0, 0)$ never falling below the
$x$-axis and ends at $(n,0)$ with up diagonal steps $(1, 1)$
weighted $1$, down diagonal steps $(1, -1)$ weighted $r_{i+1}$ and
horizontal steps $(1, 0)$ weighted $s_i$ on the line $y=i$. Then
$T_n(q)$ counts the number of these weighted paths ending at
$(n,0)$. Thus if let $\mathscr{M}_n$ denote the set of the weighted
Motzkin paths of length $n$ and
$w(\beta)=(w(\beta_1),w(\beta_2),\ldots,w(\beta_n))$ be a weighted
Motzkin path of length $n$ with
$w(\beta_i)\in\{1,s_n,r_{n+1}\}_{n\geq0}$, then we have
$$T_n(q)=\sum_{\beta\in \mathscr{M}_n}\prod_{i=1}^nw(\beta_i).$$

\end{rem}

\section{Properties of the first column }
The first column $(T_{n,0})_{n\geq0}$ of the
Stirling-Whitney-Riordan triangle $T$ has properties similar to
those of $(T_n(q))_{n\geq0}$. In this section, we will present some
properties for $(T_{n,0})_{n\geq0}$.
\begin{thm}
Let $(T_{n,0})_{n\geq0}$ be the first column of $T$.
\begin{itemize}
  \item [\rm (i)]
The ordinary generating function of $T_{n,0}$ has a Jacobi continued
  fraction expression
  $$\sum_{n\geq0}T_{n,0}t^n=\frac{1}{1-s_0t-\cfrac{r_1t^2}{1-s_1t-\cfrac{r_2t^2}{1-s_2t-\cdots}}},$$
where $s_n=a_2+a_1n+[b_1(2n+1)+b_2]\lambda$ and
$r_{n+1}=[b_1(n+1)+b_2]\lambda(b_1\lambda+a_1)(n+1)$ for $n\geq0$.
   \item [\rm (ii)]
 The sequence $(T_{n,0})_{n\geq0}$ are $\textbf{x}$-SM
and $3$-$\textbf{x}$-LCX with
$\textbf{x}=(a_1,a_2,b_1,b_2,\lambda)$.
\item [\rm (iii)]
The exponential generating function of $T_{n,0}$ is given as
 \begin{eqnarray*}
\sum_{n\geq0}T_{n,0}\frac{t^n}{n!}&=&
e^{a_2t}\left[1+\frac{b_1\lambda(1-e^{a_1t})}{a_1}\right]^{-(1+\frac{b_2}{b_1})}.
\end{eqnarray*}
   \item [\rm (iv)]
$T_{n,0}$ is a polynomial in $\lambda$ and has only real zeros.
\item [\rm (v)]
The Tur\'{a}n-type polynomial $T_{n+1,0}T_{n-1,0}-T^2_{n,0}$ is a
weakly stable polynomial in $\lambda$ for $n\geq1$.
\end{itemize}
\end{thm}

\begin{proof}
(i)
 Assume that $s_n=a_2+a_1n+[b_1(2n+1)+b_2]\lambda$,
$r_{n}=b_1(n+1)+b_2$ and $t_n=\lambda(b_1\lambda+a_1)n$ for
$n\geq0$. Let $h_k(z)=\sum_{n\geq k}T_{n,k}z^n$ for $k\geq0$. It
follows from the recurrence relation:
\begin{eqnarray*}
T_{n,k}=r_{k-1}T_{n-1,k-1}+s_k T_{n-1,k}+t_{k+1}T_{n-1,k+1}
\end{eqnarray*} that we have
\begin{eqnarray*}
h_0(z)&=&1+s_0zh_0(z)+t_1zh_1(z), \\
h_k(z)&=&r_{k-1}zh_{k-1}(z)+s_{k}zh_k(z)+t_{k+1}zh_{k+1}(z)
\end{eqnarray*}
for $k\geq1$, which imply
\begin{eqnarray*}
\frac{h_0(z)}{1}&=&\frac{1}{1-s_0z-t_1z\frac{h_1(z)}{h_0(z)}}, \\
\frac{h_1(z)}{h_0(z)}&=&\frac{r_0z}{1-s_1z-t_2z\frac{h_2(z)}{h_1(z)}},\\
&\vdots&\\
\frac{h_k(z)}{h_{k-1}(z)}&=&\frac{r_{k-1}z}{1-s_{k}z-t_{k+1}z\frac{h_{k+1}(z)}{h_k(z)}}.
\end{eqnarray*}
Thus we get
\begin{eqnarray*} \sum\limits_{n=0}^{\infty}T_{n,0}
z^n=h_0(z)=\DF{1}{1- s_0z-\DF{r_0t_1z^2}{1- s_1z-\DF{r_1t_2z^2}{1-
s_2z-\ldots}}}.
\end{eqnarray*}
If let $\mathscr{T}_n(\lambda):=T_{n,0}$, then by (i) and Theorem
\ref{thm+SM+CT} (i), we immediately get
$$\mathscr{T}_n(\lambda+q)=T_n(q)$$
for $n\geq0$. Hence we easily get (ii) and (iii) by Theorem
\ref{thm+SM+CT} (ii) and Theorem \ref{thm+EXP} (i), respectively. We
also have (iv) and (v) by Theorem \ref{thm+RZ}. This completes the
proof.
\end{proof}

\end{document}